\documentclass[12pt]{amsart}
\usepackage{amssymb}

\usepackage{amsfonts}
\usepackage{amssymb,latexsym}
\usepackage{enumerate}
\usepackage{mathrsfs}
\usepackage{url} \usepackage{color}\makeatletter
\@namedef{subjclassname@2010}{%
  \textup{2010} Mathematics Subject Classification}
\makeatother

  [2002/01/19 v2.2g %
    AMS font definitions%
  ]
\DeclareFontFamily{U}{euf}{}
\DeclareFontShape{U}{euf}{m}{n}{%
  <5><6><7><8><9>gen*eufm%
  <10><10.95><12><14.4><17.28><20.74><24.88>eufm10%
  }{}
\DeclareFontShape{U}{euf}{b}{n}{%
  <5><6><7><8><9>gen*eufb%
  <10><10.95><12><14.4><17.28><20.74><24.88>eufb10%
  }{}

  [2002/01/19 v2.2g %
    AMS font definitions%
  ]
\DeclareFontFamily{U}{msb}{}
\DeclareFontShape{U}{msb}{m}{n}{%
  <5><6><7><8><9>gen*msbm%
  <10><10.95><12><14.4><17.28><20.74><24.88>msbm10%
  }{}

  [2002/01/19 v2.2g %
    AMS font definitions%
  ]
\DeclareFontFamily{U}{msa}{}
\DeclareFontShape{U}{msa}{m}{n}{%
  <5><6><7><8><9>gen*msam%
  <10><10.95><12><14.4><17.28><20.74><24.88>msam10%
  }{}

\newtheorem{theorem}{Theorem}[section]
\newtheorem{lemma}[theorem]{Lemma}

\newtheorem{corollary}[theorem]{Corollary}

\theoremstyle{definition}

\newtheorem{remark}[theorem]{Remark}

\numberwithin{equation}{section}

\numberwithin{equation}{section} 

\begin{document}

\title[]{Appell-Carlitz numbers}

\author{Su Hu}
\address{Department of Mathematics, South China University of Technology, Guangzhou 510640, China}
\email{mahusu@scut.edu.cn}

\author{Min-Soo Kim}
\address{Department of Mathematics Education, Kyungnam University, Changwon, Gyeongnam 51767, Republic of Korea}
\email{mskim@kyungnam.ac.kr}

%\thanks{This work was financially supported by KRF 2003-041-C20009}


\maketitle

\begin{abstract}
In this paper, we introduce the concept of the (higher order) Appell-Carlitz numbers
which unifies the definitions of several special numbers in positive characteristic,
such as the Bernoulli-Carlitz numbers and the Cauchy-Carlitz numbers.
Their generating function
is named Hurwitz series in the function field arithmetic (\cite[p. 352, Def. 9.1.4]{Go3}). By using Hasse-Teichm\"uller derivatives, we also obtain several properties of the (higher order) Appell-Carlitz numbers,
including a recurrence formula, two closed forms expressions, and a determinant expression.

The recurrence formula implies Carlitz's recurrence formula for Bernoulli-Carlitz numbers.
Two closed from expressions implies the corresponding results for Bernoulli-Carlitz and Cauchy-Carlitz numbers .
The determinant expression implies the  corresponding results for Bernoulli-Carlitz and Cauchy-Carlitz numbers,
which are analogues of the classical determinant expressions of Bernoulli
and Cauchy numbers stated in an article by Glaisher in 1875.
\end{abstract}

\section{Introduction}\label{intro}

The Bernoulli numbers $B_{n}\in\mathbb{Q}~ (n=0,1,2,\ldots)$ are defined by the generating function
\begin{equation}\label{Ber}
\frac{t}{e^{t}-1}=\sum_{n=0}^{\infty} B_{n}\frac{t^{n}}{n!}.
\end{equation}
The Bernoulli numbers may also be defined  by the recursive formula
\begin{eqnarray}\label{Ber-recursivef}
B_{0} =1, \quad  B_n
=-\sum_{j=0}^{n-1}\frac{n!}{j!(n+1-j)!}B_{j}\quad \text{for }n \geq 1,
\end{eqnarray}
which can be obtained by comparing the coefficients in the expansion of $t ={(e^t -1)}
\sum_{n=0}^{\infty}\frac{B_n}{n!}t^n.$
The Bernoulli numbers have many applications in modern number theory,
such as the Eisenstein series in modular forms (see \cite{Apostol}), and the arithmetic of algebraic number fields, especially Kummer's notion of
regularity and the  class number of $p$th cyclotomic fields (see \cite[p. 62, Theorem 5.16]{Washington}).

It is well-known that there exist close analogues between the rational number field $\mathbb{Q}$
and the rational function fields $\mathbb{F}_{r}(T)$ over a finite field
$\mathbb{F}_{r}$ (see \cite{Go2}). In 1935, Carlitz~\cite{Carlitz} gave an analogue of Bernoulli numbers for rational function field $\mathbb{F}_{r}(T),$
denoted by $BC_{n}$, which is now known as the Bernoulli-Carlitz numbers. In subsequent works, he also found many interesting properties of them, including the analogue of
the well-known von Staudt--Clausen theorem (see \cite{Carlitz2, Carlitz3} and \cite{Lara}).
The definition of Bernoulli-Carlitz numbers is as follows.
Let $[i] =T^{r^i}-T$, $D_{i}=[i]{[i-1]}^r \cdots {[1]}^{r^{i-1}}$
with $D_{0}=1$. The Carlitz exponential is defined by
$$e_C(z)=\sum_{j=0}^\infty\frac{z^{r^j}}{D_j}.$$
For a nonnegative integer $i$ with $r$-ary expansion
$i=\sum_{j=0}^{m}c_{j}r^{j}~(0\leq c_{j}<r),$
the Carlitz factorial $\Pi(i)$ is defined by
$$\Pi(i)=\prod_{j=0}^{m}D_{j}^{c_j}.$$
In analogue with (\ref{Ber}), the Bernoulli-Carlitz numbers $BC_{n}\in\mathbb{F}_{r}(T)~ (n=0,1,2,\ldots)$
are defined by the generating function
\begin{eqnarray}\label{defofBC}
\frac{z}{e_{C}(z)}=\sum_{n=0}^\infty\frac{BC_n}{\Pi(n)}z^n.
\end{eqnarray}
By comparing the coefficients in the expansion of $z=e_{C}(z)\sum_{n=0}^\infty\frac{BC_n}{\Pi(n)}z^n,$
Carlitz found the following recursive formula of the Bernoulli-Carlitz numbers $BC_{n}$
which are  analogues of (\ref{Ber-recursivef})
\begin{equation}\label{Ber-Car-recursivef}
BC_0=1, \quad BC_{n}=-\sum_{j=1}^{[\log_{r}(n+1)]}
\frac{\Pi(n)}{\Pi(r^j) \Pi(n+1-r^j)}BC_{n+1-r^j} \quad \text{for }n \geq 1,
\end{equation}
where $[\cdot]$ is the greatest integer function. As in the classical case, the Bernoulli-Carlitz numbers  have many deep connections with the arithmetic of function fields, especially  the class
groups of cyclotomic function fields (see \cite[Sec. 9.2]{Go3}, \cite[Sec. 5.2]{Thakur} or \cite{Gekeler1, Gekeler2}).

For $\ell\in\mathbb{N},$ in 1924, N\"orlund \cite{Norlund} defined the higher order Bernoulli numbers $B_{n}^{(\ell)}\in\mathbb{Q}~ (n=0,1,2,\ldots)$
by the generating function
\begin{equation}\label{Ber-h}
\left(\frac{t}{e^{t}-1}\right)^{\ell}=\sum_{n=0}^{\infty} B_{n}^{(\ell)}\frac{t^{n}}{n!},
\end{equation}
and in 2005, Jeong, Kim and Son~\cite{JKS} defined the higher order
Bernoulli-Carlitz numbers $BC_{n}^{(\ell)}\in\mathbb{F}_{r}(T)~ (n=0,1,2,\ldots)$
by the generating function
\begin{eqnarray}\label{BC-h}
\left(\frac{z}{e_{C}(z)}\right)^{\ell}=\sum_{n=0}^\infty\frac{BC_{n}^{(\ell)}}{\Pi(n)}z^n.
\end{eqnarray}
Letting $l=1$ in (\ref{Ber-h}) and (\ref{BC-h}),
we recover the Bernoulli numbers $B_{n}$ and the Bernoulli-Carlitz numbers
$BC_{n},$ respectively.

The classical Cauchy numbers $c_{n}\in\mathbb{Q}~ (n=0,1,2,\ldots)$ are defined by the generating function
\begin{equation}\label{Cauchy}
\frac{t}{\log(1+t)}=\sum_{n=0}^{\infty} c_{n}\frac{t^{n}}{n!}
\end{equation} (see \cite{Glaisher}).

Let $L_{i}=[i][i-1]\cdots[1]~(i\geq 1)$ with $L_{0}=1,$ let
\begin{equation}\label{carlitzlog}
\log_C(z)=\sum_{i=0}^\infty(-1)^i\frac{z^{r^i}}{L_i}
\end{equation}
be the Carlitz logarithm.
In 2016, Kaneko and Komatsu~\cite{KK} defined the Cauchy-Carlitz numbers $CC_{n}~(n=0,1,2,\ldots)$ by the generating function
\begin{equation}\label{caucarlitz}
\frac{z}{\log_C(z)}=\sum_{n=0}^\infty\frac{CC_n}{\Pi(n)}z^n
\end{equation}
(see \cite[p. 240, (12)]{KK}),
and for $\ell\in\mathbb{N},$ they also defined the higher order Cauchy-Carlitz numbers $CC_{n}^{(\ell)} (n=0,1,2,\ldots)$ by
\begin{eqnarray}\label{CC-h}
\left(\frac{z}{\log_C(z)}\right)^{\ell}=\sum_{n=0}^\infty\frac{CC_{n}^{(\ell)}}{\Pi(n)}z^n
\end{eqnarray}
(see \cite[p. 249, (28)]{KK}).

Recently, in order to unify the definitions of several special numbers in the classical setting
such as the (higher order) Bernoulli numbers and the (higher order) Cauchy numbers, Hu and Komatsu~\cite{HKo}
introduced the concept of the related numbers of higher order Appell polynomials.
Their definition is as follows.

Let $\mathbb{C}$ be the field of complex numbers,
let $S(t)=\sum_{n=0}^{\infty}a_{n}\frac{t^{n}}{n!}$ be any  formal power series in $\mathbb{C}[[t]]$ and $a_{0}\neq 0,$
the Appell polynomials $A_{n}(z)$ are defined by the generating function
\begin{equation}\label{Appell}
S(t)e^{zt}=\sum_{n=0}^{\infty}A_{n}(z)\frac{t^{n}}{n!}
\end{equation}
(see \cite{Appell}).
Since $a_{0}\neq 0,$ there exists the formal power series (for some $d_n\in\mathbb C$)
\begin{equation}\label{ft}
f(t)=\frac{1}{S(t)}=\sum_{n=0}^{\infty}d_{n}\frac{t^{n}}{n!}
\end{equation}
in $\mathbb{C}[[t]]$, and (\ref{Appell}) becomes
\begin{equation}\label{Appell2}
\frac{e^{zt}}{f(t)}=\sum_{n=0}^{\infty}A_{n}(z)\frac{t^{n}}{n!}.
\end{equation}
For $\ell\in\mathbb{N},$ we can also define the higher order Appell polynomials by the generating function
\begin{equation}\label{Appellh}
\frac{e^{zt}}{(f(t))^{\ell}}=\sum_{n=0}^{\infty}A_{n}^{(\ell)}(z)\frac{t^{n}}{n!}
\end{equation}
(see \cite[Th\'eor\`eme 1.1]{BBZ}).
As in the classical case, $a_{n}^{(\ell)}=A_{n}^{(\ell)}(0)$ is defined to be the related numbers of higher order Appell polynomials, that is,
\begin{equation}\label{Appellhn}
\frac{1}{(f(t))^{\ell}}=\sum_{n=0}^{\infty}a_{n}^{(\ell)}\frac{t^{n}}{n!}
\end{equation}
and $a_{n}=a_{n}^{(1)}$ the related numbers of Appell polynomials (see \cite[p. 3, (6)]{HKo}).

In (\ref{Appellhn}), let $f(t)=\frac{e^{t}-1}{t}$ and $\frac{\log(1+t)}{t},$
we obtain the (higher order) Bernoulli numbers and the (higher order) Cauchy numbers, respectively.

To unify the definitions of several special numbers in positive characteristic,
such as the  Bernoulli-Carlitz numbers and the  Cauchy-Carlitz numbers, we here define the Appell-Carlitz numbers to be a sequence $\{AC_{n}\}_{n=0}^{\infty}$ in
$k=\mathbb{F}_{r}(T)$ with a normalization $AC_{0}=1.$
Let $S(z)\in k((z))$ be the generating function of $\{AC_{n}\}_{n=0}^{\infty},$ that is,
\begin{eqnarray}\label{def-Sz}
S(z)=\sum_{n=0}^\infty\frac{AC_n}{\Pi(n)}z^n.
\end{eqnarray}
If $S(z)=\left(\frac{z}{e_C(z)}\right)^\ell$ and $S(z)=\left(\frac{z}{\log_C(z)}\right)^\ell,$ then we obtain the higher order Bernoulli-Carlitz numbers
and the  higher order  Cauchy-Carlitz numbers, respectively. It needs to mention that in Goss's book \cite[p. 352, Def. 9.1.4]{Go3}, the above generating function $S(z)$
is named  as the Hurwitz series.

For $\ell\in\mathbb{N}$, we may also define the higher order Appell-Carlitz numbers $AC_{n}^{(\ell)}(z)~(n=0,1,2,\ldots)$ as the generating function
\begin{eqnarray}\label{def-hapn}
(S(z))^\ell=\sum_{n=0}^\infty\frac{AC_n^{(\ell)}}{\Pi(n)}z^n.
\end{eqnarray}
Denote by $f(z)=\frac1{S(z)},$ we have
\begin{eqnarray}\label{def-hapnf}
\frac{1}{(f(z))^\ell}=\sum_{n=0}^\infty\frac{AC_n^{(\ell)}}{\Pi(n)}z^n,
\end{eqnarray}
which is an analogue of (\ref{Appellhn}) in $\mathbb{F}_{r}(T).$ 
It should be noted that the definitions of the Appell-Carlitz numbers and their higher order
counterparts depend on the series $S(z).$

\section{Main results and their corollaries}\label{Sec.2}

``In mathematics, a closed form is a mathematical expression that can be evaluated
in a finite number of operations. It may contain constants, variables, four arithmetic
operations, and elementary functions, but usually no limit." (See \cite[p. 91]{Qi2016}).
During the recent years, there are many results concerning closed form expressions for special numbers and polynomials in characteristic 0 case, such as Bernoulli, Euler, Cauchy, Apostol--Bernoulli,
hypergeometric Bernoulli numbers and polynomials, see \cite{CK, Dagali, HK1, HKo, Qi2016} and the references  therein.

In this paper, we shall address our attention to the characteristic $p$ case and  obtain several properties of the (higher order) Appell-Carlitz numbers,
including a recurrence formula, two closed form expressions, a determinant expression. The recurrence formula (Theorem \ref{Theorem1}) implies Carlitz's recurrence formula for Bernoulli-Carlitz numbers
(see (\ref{Ber-Car-recursivef}) above). Two closed from expressions (Theorems \ref{Theorem2} and \ref{Theorem3})
implies the corresponding results for Bernoulli-Carlitz and Cauchy-Carlitz numbers (see Corollaries \ref{Corollary2}, \ref{Corollary3},
\ref{Corollary4} and \ref{Corollary5} below).
The determinant expression (Theorem \ref{Theorem4}) implies the  corresponding results
for Bernoulli-Carlitz and Cauchy-Carlitz numbers (see Corollaries \ref{Corollary6} and \ref{Corollary7} below).

Suppose that $f(z)=\frac1{S(z)}$ has the following power series expansion
\begin{eqnarray}\label{def-f(z)}
f(z)=\sum_{n=0}^\infty\lambda_n z^n.
\end{eqnarray}
Then we have the following recurrence formula for the higher order Appell-Carlitz numbers.

\begin{theorem}[Recurrence formula for higher order Appell-Carlitz numbers]\label{Theorem1}
$$AC_{m}^{(\ell)}=-\Pi(m)\sum_{i=0}^{m-1}\frac{AC_{i}^{(\ell)}}{\Pi(i)}D_\ell(m-i)$$
with $AC_{0}^{(\ell)}=1,$ where
\begin{equation}\label{D-ell}
D_\ell(e)=\sum_{i_1+\cdots+i_\ell=e\atop i_1,\dots,i_\ell\ge 0}\lambda_{i_1}\cdots\lambda_{i_\ell}.
\end{equation}
\end{theorem}

Letting $\ell=1$ in the above result, we get a recurrence formula for Appell-Carlitz numbers.

\begin{corollary}[Recurrence formula for Appell-Carlitz numbers]\label{Corollary1}
$$
AC_{m}=-\Pi(m)\sum_{i=0}^{m-1}\frac{AC_{i}}{\Pi(i)}\lambda_{m-i}.
$$
\end{corollary}

\begin{remark}\label{Remark1}
In the case of Bernoulli-Carlitz numbers,
we have $$f(z)=\frac{e_{C}(z)}{z}=\sum_{j=0}^{\infty}\frac{z^{r^{j}-1}}{D_{j}}.$$
Define
\begin{equation}\label{(8)}
\delta_e^*=\begin{cases}
\frac1{D_n} &\text{if $e=r^n-1$ for some }n \\
0 &\text{if $e\neq r^n-1$ for any }n,
\end{cases}
\end{equation}
then comparing with (\ref{def-f(z)}), we have
\begin{equation}\label{(10)}\lambda_{j}=\delta_{j}^{*}\end{equation} for $j=0,1,2,\ldots$.
By Corollary \ref{Corollary1} and (\ref{(10)}), we obtain Carlitz's recurrence formula for Bernoulli-Carlitz numbers (see (\ref{Ber-Car-recursivef}) above)
\begin{equation}
\begin{aligned}
BC_{m}&=-\Pi(m)\sum_{i=0}^{m-1}\frac{BC_{i}}{\Pi(i)}\delta_{m-i}^{*}\\
&=-\Pi(m)\sum_{i=0}^{m-1}\frac{BC_{m-i}}{\Pi(m-i)}\delta_{i}^{*}\\
&=-\Pi(m)\sum_{j=0}^{[\log_{r}(m+1)]}\frac{BC_{m+1-r^{j}}}{\Pi(m+1-r^{j})}\delta_{r^{j}-1}^{*}\\
&=-\sum_{j=1}^{[\log_{r}(m+1)]}
\frac{\Pi(m)}{\Pi(r^j) \Pi(m+1-r^j)}BC_{m+1-r^j},
\end{aligned}
\end{equation}
since $D_{j}=\Pi(r^{j}),$ for $m\geq1.$
\end{remark}

We also have a closed form expression for  the higher order Appell-Carlitz numbers.

\begin{theorem}[Closed form expression for  higher order Appell-Carlitz numbers]\label{Theorem2}
For $m\ge 1$, we have
$$
AC_{m}^{(\ell)}=\Pi(m)\sum_{k=1}^m (-1)^{k}\sum_{e_1+\cdots+e_k=m\atop e_1,\dots,e_k\ge1}D_\ell(e_1)\cdots D_\ell(e_k),
$$
where $D_\ell(e)$ are given in $(\ref{D-ell})$.
\end{theorem}

Letting $\ell=1$ in the above result, we have a closed form expression for Appell-Carlitz numbers.

\begin{corollary}[Closed form expression for Appell-Carlitz numbers]\label{Corollary2}
For $m\ge 1$, we have
$$
AC_{m}=\Pi(m)\sum_{k=1}^{m}(-1)^{k}\sum_{e_1+\cdots+e_k=m\atop e_1,\dots,e_k\ge1}\lambda_{e_1}\cdots \lambda_{e_k}.
$$
\end{corollary}

Then by (\ref{(10)}), we have
$$\begin{aligned}
BC_m&=\Pi(m)\sum_{k=1}^m (-1)^{k}\sum_{e_1+\cdots+e_k=m\atop e_1,\dots,e_k\ge1}\delta_{e_1}^*\cdots \delta_{e_k}^* \\
&=\Pi(m)\sum_{k=1}^m (-1)^{k}\sum_{r^{i_1}+\cdots+r^{i_k}=m+k\atop r^{i_1},\dots,r^{i_k}>1}
\frac1{D_{i_1}}\cdots\frac1{D_{i_k}}.
\end{aligned}$$
Since $D_{i}=\Pi(r^{i})$, we have the following closed form expression for Bernoulli-Carlitz numbers by Jeong, Kim and Son (see \cite[p. 63, Theorem 4.2]{JKS}).

\begin{corollary}[Closed form expression for Bernoulli-Carlitz numbers]\label{Corollary3}
For $m\ge 1$, we have $$BC_m=\Pi(m)\sum_{k=1}^{m} (-1)^{k}\sum_{r^{i_1}+\cdots+r^{i_k}=m+k\atop r^{i_1},\dots,r^{i_k}>1}
\frac1{\Pi(r^{i_1})}\cdots\frac1{\Pi(r^{i_k})}.$$
\end{corollary}

More generally, from Theorem \ref{Theorem2} and (\ref{(10)}), we may also recover the following closed form
expression for higher order Bernoulli-Carlitz numbers (see \cite[p. 65, Proposition 4.5]{JKS}).
We would like to refer Thakur's book \cite[p. 145, the second last paragraph]{Thakur} on a discussion of this formula.

\begin{corollary}[Closed form expression for higher order Bernoulli-Carlitz numbers]\label{Corollary4}
For $m\ge 1$, we have
$$ BC_{m}^{(\ell)}
=\Pi(m) \sum_{j=1}^{m}(-1)^j  \sum_{\substack{ i_1, \ldots, i_{j} \geq 1 \\
i_1 + \cdots + i_j =m}} M^{(\ell)}(i_1 )\cdots M^{(\ell)}(i_j ),$$
where for each $i,$
$$M^{(\ell)}(i) :=\sum_{\substack{e_1,\ldots, e_{\ell} \geq 0 \\
r^{e_1}  + \cdots + r^{e_\ell} =i}} \frac{1}{\Pi(r^{e_1})\Pi(r^{e_2})
\cdots \Pi(r^{e_\ell}) }.$$
\end{corollary}

In the case of Kaneko and Komatsu's Bernoulli-Carlitz numbers,
we have $$f(z)=\frac{\log_{C}(z)}{z}=\sum_{j=0}^{\infty}(-1)^{j}\frac{z^{r^{j}-1}}{L_{j}}.$$
Define
\begin{equation}\label{(9)}
\delta_e^{**}=\begin{cases}
(-1)^{n}\frac1{L_{n}} &\text{if $e=r^n-1$ for some }n \\
0 &\text{if $e\neq r^n-1$ for any }n,
\end{cases}
\end{equation}
then comparing with (\ref{def-f(z)}), we have
\begin{equation}\label{(11)}\lambda_{j}=\delta_{j}^{**}\end{equation} for $j=0,1,2,\ldots.$

From Theorem \ref{Theorem2} and (\ref{(10)}), we also recover the following closed form
expression for Kaneko and Komatsu's higher order Cauchy-Carlitz numbers (see \cite[p. 249, Proposition 6]{KK}).

\begin{corollary}[Closed form expression for higher order Cauchy-Carlitz numbers]\label{Corollary5}
For $m\ge 1$, we have
$$ CC_{m}^{(\ell)}
=\Pi(m) \sum_{j=1}^{m}(-1)^j  \sum_{\substack{ i_1, \ldots, i_{j} \geq 1 \\
i_1 + \cdots + i_j =m}} M^{(\ell)}(i_1 )\cdots M^{(\ell)}(i_j ),$$ where for each $i,$
$$M^{(\ell)}(i) :=\sum_{\substack{j_1,  \ldots, j_{\ell} \geq 0 \\
r^{j_1} + \cdots + r^{j_\ell} =i}}\frac{(-1)^{j_1+\cdots+j_\ell}}{L_{j_1}\cdots L_{j_\ell}}.$$
\end{corollary}

Generalizing Jeong, Kim and Son's result for Bernoulli-Carlitz numbers \cite[p. 63, Theorem 4.1]{JKS},
we get another closed form expression for Appell-Carlitz numbers.

\begin{theorem}[Another closed form expression for Appell-Carlitz numbers]\label{Theorem3}
For $m\geq 1$, we have
$$\begin{aligned}
AC_{m}
&=\Pi(m) \sum_{j=1}^{m} (-1)^j  \sum_{\substack{ i_1, \ldots, i_m \geq 0 \\
i_1 +\cdots+ i_m =j\\
i_1 +2i_2 + \cdots+ mi_m =m } }  \binom{j}{i_1,  \ldots, i_m}
\lambda_1^{i_1}  \lambda_2^{i_2}  \cdots \lambda_m^{i_m}.
\end{aligned}$$
\end{theorem}

We have a determinant expression of the higher order Appell-Carlitz numbers.

\begin{theorem}[Determinant expression of higher order Appell-Carlitz numbers]\label{Theorem4}
For $m\ge 1$, we have
$$
AC_{m}^{(\ell)}=(-1)^m \Pi(m)\left|
\begin{array}{ccccc}
D_\ell(1)&1&&&\\
D_\ell(2)&D_\ell(1)&&&\\
\vdots&\vdots&\ddots&1&\\
D_\ell(n-1)&D_\ell(n-2)&\cdots&D_\ell(1)&1\\
D_\ell(n)&D_\ell(n-1)&\cdots&D_\ell(2)&D_\ell(1)
\end{array}
\right|,$$
where $D_\ell(e)$ are given in $(\ref{D-ell})$.
\end{theorem}

Letting $\ell=1$ in the above result, we have the following determinant expression for the related numbers of  Appell-Carlitz numbers.

\begin{corollary}[Determinant expression of Appell-Carlitz numbers]
For $m\ge 1$, we have
\begin{align*}
AC_{m}
=(-1)^m \Pi(m)\left|
\begin{array}{ccccc}
\lambda_{1}&1&&&\\
\lambda_2&\lambda_{1}&&&\\
\vdots&\vdots&\ddots&1&\\
\lambda_{m-1}&\lambda_{m-2}&\cdots&\lambda_{1}&1\\
\lambda_{m}&\lambda_{m-1}&\cdots&\lambda_{2}&\lambda_{1}
\end{array}
\right|.
\end{align*}
\end{corollary}

Then by (\ref{(10)}), we obtain a determinant expression of Bernoulli-Carlitz numbers.

\begin{corollary}[Determinant expression of Bernoulli-Carlitz numbers]\label{Corollary6}
For  $m\geq 1,$ we have
\begin{equation}
\begin{aligned}
&BC_{m}
=(-1)^m \Pi(m)\left|
\begin{array}{ccccc}
\delta_1^*&1&&&\\
\delta_2^*&\delta_1^*&&&\\
\vdots&\vdots&\ddots&1&\\
\delta_{m-1}^*&\delta_{m-2}^*&\cdots&\delta_1^*&1\\
\delta_m^*&\delta_{m-1}^*&\cdots&\delta_2^*&\delta_1^*
\end{array}
\right|,
\end{aligned}
\label{det:ghbn}
\end{equation}
where $$\delta_e^*=\begin{cases}
\frac1{D_n} &\text{if $e=r^n-1$ for some }n \\
0 &\text{if $e\neq r^n-1$ for any }n.
\end{cases}$$
\end{corollary}

\begin{remark}
Since $D_{n}$ equals to $\Pi(r^{n})$, the Carlitz factorial, the above result is an analogue of the following classical determinant expression of Bernoulli numbers $B_{m}$ by Glaisher in 1875 (see \cite[p. 53]{Glaisher}):
\begin{equation}\label{Berd-g}
B_m=(-1)^{m} m!\left|
\begin{array}{ccccc}
\frac{1}{2!}&1&&&\\
\frac{1}{3!}&\frac{1}{2!}&&&\\
\vdots&\vdots&\ddots&1&\\
\frac{1}{m!}&\frac{1}{(m-1)!}&\cdots&\frac{1}{2!}&1\\
\frac{1}{(m+1)!}&\frac{1}{m!}&\cdots&\frac{1}{3!}&\frac{1}{2!}
\end{array}
\right|.
\end{equation}
\end{remark}

Similarly, by (\ref{(11)}), we obtain the following determinant expression of  Cauchy-Carlitz numbers.
\begin{corollary}[Determinant expression of Cauchy-Carlitz numbers]\label{Corollary7}
For $m\geq 1,$ we have
\begin{equation}
\begin{aligned}
&CC_{m}
=(-1)^m \Pi(m)\left|
\begin{array}{ccccc}
\delta_1^{**}&1&&&\\
\delta_2^{**}&\delta_1^{**}&&&\\
\vdots&\vdots&\ddots&1&\\
\delta_{m-1}^{**}&\delta_{m-2}^{**}&\cdots&\delta_1^{**}&1\\
\delta_m^{**}&\delta_{m-1}^{**}&\cdots&\delta_2^{**}&\delta_1^{**}
\end{array}
\right|,
\end{aligned}
\label{det:ghcn}
\end{equation}
where $$\delta_e^{**}=\begin{cases}
(-1)^{n}\frac1{L_{n}} &\text{if $e=r^n-1$ for some }n \\
0 &\text{if $e\neq r^n-1$ for any }n,
\end{cases}$$
\end{corollary}

\begin{remark}
Since the classical Cauchy numbers are defined by the generating function
\begin{equation}\label{Cauchy}
\frac{t}{\log(1+t)}=\sum_{n=0}^{\infty} c_{n}\frac{t^{n}}{n!}
\end{equation} (see \cite{Glaisher}),
by applying the power series expansion of \begin{equation}\label{power}f(t)=\frac{\log(1+t)}{t}=\sum_{n=0}^{\infty}(-1)^{n}\frac{t^{n}}{n+1}
\end{equation}
 to \cite[Theorem 3]{HKo},
we get the following determinant expression of Cauchy numbers
\begin{equation}\label{C-deter+}
\begin{aligned}
c_{m}
&= (-1)^{m}m!\left|
\begin{array}{ccccc}
-\frac{1}{2}&1&&&\\
\frac{1}{3}&-\frac{1}{2}&&&\\
\vdots&\vdots&\ddots&1&\\
\frac{(-1)^{m-1}}{m}&\frac{(-1)^{m-2}}{m-1}&\cdots&-\frac{1}{2}&1\\
\frac{(-1)^{m}}{m+1}&\frac{(-1)^{m-1}}{m}&\cdots&\frac{1}{3}&-\frac{1}{2}
\end{array}\right|.
\end{aligned}
\end{equation}
This is equivalent  to Glaisher's following determinant expression in 1875 (see \cite[p. 50]{Glaisher}):
\begin{equation}\label{C-deter}
\begin{aligned}
c_{m}
&= m!\left|
\begin{array}{ccccc}
\frac{1}{2}&1&&&\\
\frac{1}{3}&\frac{1}{2}&&&\\
\vdots&\vdots&\ddots&1&\\
\frac{1}{m}&\frac{1}{m-1}&\cdots&\frac{1}{2}&1\\
\frac{1}{m+1}&\frac{1}{m}&\cdots&\frac{1}{3}&\frac{1}{2}
\end{array}\right|
\end{aligned}
\end{equation}
if considering the generating function
\begin{equation}\label{Cauchy}
\frac{-t}{\log(1-t)}=\sum_{n=0}^{\infty}(-1)^{n}c_{n}\frac{t^{n}}{n!}
\end{equation}
and applying the power series expansion
$$f(t)=\frac{\log(1-t)}{-t}=\sum_{n=0}^{\infty}\frac{t^{n}}{n+1}$$ to \cite[Theorem 3]{HKo}.
By comparing the power series expansion of the Carlitz logarithm
$$
\log_C(z)=\sum_{i=0}^\infty(-1)^i\frac{z^{r^i}}{L_i}
$$
and the power series expansion of the classical logarithm
$$\log(1+t) =\sum_{n=1}^{\infty}(-1)^{n-1}\frac{t^{n}}{n},$$
we have seen the analogue between (\ref{det:ghcn}) and (\ref{C-deter+}).
\end{remark}

\section{Hasse-Teichm\"uller derivatives (\cite[Section 2]{HKo} and \cite[Section 2]{JKS})}\label{Ha-Te}
Since $n!=0$ in a field of characteristic $p$ if $n\geq p$, and $\frac{d}{dt}(t^{n})=0$ if $p$ divides $n$,
the classical differential calculus faces essential difficulties in positive characteristics. In 1936, Hasse \cite{Hasse} introduced the concept of hyperdifferentials to overcome these, now known as the Hasse-Teichm\"uller derivatives.
In this section, we shall recall the definition and basic properties of these derivatives
which serves as the main tool for our proof.

Let $\mathbb{F}$ be a field of any characteristic, $\mathbb{F}[[z]]$ the ring of formal power series in one variable $z$, and $\mathbb{F}((z))$ the field of Laurent series in $z$. Let $m$ be a nonnegative integer. The Hasse-Teichm\"uller derivative $H^{(m)}$ of order $m$ is defined by
\begin{equation}\label{HT}
H^{(m)}\left(\sum_{n=R}^{\infty} c_n z^n\right)
=\sum_{n=R}^{\infty} c_n \binom{n}{m}z^{n-m}
\end{equation}
for $\sum_{n=R}^{\infty} c_n z^n\in \mathbb{F}((z))$,
where $R$ is an integer and $c_n\in\mathbb{F}$ for any $n\geq R$. Note that $\binom{n}{m}=0$ if $n<m$.

The Hasse-Teichm\"uller derivatives satisfy the product rule \cite{Teich}, the quotient rule \cite{GN} and the chain rule \cite{Hasse}.
One of the product rules can be described as follows.

\begin{lemma}[{\cite{Teich,JKS}}]\label{productrule2}
For $f_i\in\mathbb F[[z]]$ $(i=1,\dots,k)$ with $k\ge 2$ and for $m\ge 1$, we have
$$H^{(m)}(f_1\cdots f_k)=\sum_{i_1+\cdots+i_k=m\atop i_1,\dots,i_k\ge 0}H^{(i_1)}(f_1)\cdots H^{(i_k)}(f_k).$$
\end{lemma}

The quotient rules can be described as follows.

\begin{lemma}[\cite{GN,JKS}]
For $f\in\mathbb F[[z]]\backslash \{0\}$ and $m\ge 1$,
we have
\begin{align}
H^{(m)}\left(\frac{1}{f}\right)&=\sum_{k=1}^m\frac{(-1)^k}{f^{k+1}}\sum_{i_1+\cdots+i_k=m\atop i_1,\dots,i_k\ge 1}H^{(i_1)}(f)\cdots H^{(i_k)}(f)
\label{quotientrule1}
\\
&=\sum_{k=1}^m\binom{m+1}{k+1}\frac{(-1)^k}{f^{k+1}}\sum_{i_1+\cdots+i_k=m\atop i_1,\dots,i_k\ge 0}H^{(i_1)}(f)\cdots H^{(i_k)}(f).
\label{quotientrule2}
\end{align}
\label{quotientrules}
\end{lemma}

\begin{lemma}[\cite{Teich,JKS}]\label{quotientrule3}
For $f\in\mathbb F[[z]]$ and for $m\ge 1,$ $j\geq2,$ we have
$$\begin{aligned}
H^{(m)}({f^j})
&=\sum_{k=1}^{j} f^{j-k} \sum_{\substack{ i_1, \ldots, i_m \geq 0 \\ i_1 + \cdots+ i_m =k \\ i_1 +2i_2 + \cdots+ mi_m =m } }
\frac{j(j-1) \cdots (j-k+1)}{i_1 ! \cdots i_m !} \\
 &\quad\times (H^{(1)}(f))^{i_1}  \cdots (H^{(m)}(f))^{i_m}.
\end{aligned}$$
\end{lemma}

\section{Proofs of the main results}

In this section, we shall proof  our main results which have been introduced in Section \ref{Sec.2}.

\begin{lemma}\label{prp0h-r}
For $m\geq 1$, we have
$$\sum_{i_{\ell+1}=0}^m \frac{AC_{i_{\ell+1}}^{(\ell)}}{\Pi(i_{\ell+1})}
\sum_{i_1+\cdots+i_\ell=m-i_{\ell+1}\atop i_1,\dots,i_\ell\ge 0}{\lambda_{i_{1}}\cdots \lambda_{i_{\ell}}}=0.$$
\end{lemma}
\begin{proof}
Put $f(z)=\frac1{S(z)}.$
From (\ref{def-hapn}) and (\ref{def-f(z)}), we have
\begin{equation}\label{pf-m}
1=\left(f(z)\right)^\ell\left(S(z)\right)^\ell\\
=\left(\sum_{n=0}^\infty\lambda_n z^n\right)^\ell
\left( \sum_{n=0}^\infty\frac{AC_n^{(\ell)}}{\Pi(n)}z^n\right).
\end{equation}
Applying the Hasse-Teichm\"uller derivative $H^{(m)}$ of order $m$ to (\ref{pf-m}), we have
\begin{equation}\label{pf-m-1}
H^{(m)}\left.\left[\left(f(z)\right)^\ell
\left(S(z)\right)^\ell\right]\right|_{z=0}
=\left.H^{(m)}(1)\right|_{z=0}=0.
\end{equation}
Note that for $j=1,2,\ldots,\ell$, by the definition of the  Hasse-Teichm\"uller derivative (\ref{HT}), we have
$$\begin{aligned}
\left.H^{(i_j)}(f(z))\right|_{z=0}&=\left.H^{(i_j)}\left(\sum_{n=0}^\infty\lambda_n z^n\right)\right|_{z=0} \\
&=\left.\sum_{n=i_j}^{\infty}\lambda_n\binom{n}{i_j}z^{n-i_j}\right|_{z=0} \\
&=\lambda_{i_j}
\end{aligned}$$
and
$$\begin{aligned}
\left.H^{(i_{\ell+1})}\left[\left(S(z)\right)^{\ell}\right]\right|_{z=0}&=\left.H^{(i_{\ell+1})}\left(\sum_{n=0}^\infty\frac{AC_n^{(\ell)}}{\Pi(n)}z^n\right)\right|_{z=0} \\
&=\left.\sum_{n=i_{\ell+1}}^{\infty}\frac{AC_n^{(\ell)}}{\Pi(n)}\binom{n}{i_{\ell+1}}z^{n-i_{\ell+1}}\right|_{z=0} \\
&=\frac{AC_{i_{\ell+1}}^{(\ell)}}{\Pi(i_{\ell+1})}.
\end{aligned}$$
Then by Lemma \ref{productrule2}, we have
\begin{equation}\label{pf-m-2}
\begin{aligned}
&H^{(m)}\left.\left[\left(f(z)\right)^\ell
\left(S(z)\right)^{\ell}\right]\right|_{z=0}\\
&=\sum_{i_1+\cdots+i_{\ell+1}=m\atop i_1,\dots,i_k\ge 0}\left.
H^{(i_1)}\left(f(z)\right)\right|_{z=0}\cdots \left.H^{(i_\ell)}\left(f(z)\right)\right|_{z=0} \left.H^{(i_{\ell+1})}\left[\left(S(z)\right)^{\ell}\right]\right|_{z=0}\\
&=\sum_{i_1+\cdots+i_{\ell+1}=m\atop i_1,\dots,i_k\ge 0}
\lambda_{i_1}\cdots\lambda_{i_\ell}\frac{AC_{i_{\ell+1}}^{(\ell)}}{\Pi(i_{\ell+1})}.
\end{aligned}
\end{equation}
Comparing with (\ref{pf-m-1}), we get
\begin{equation}
\sum_{i_1+\cdots+i_{\ell+1}=m\atop i_1,\dots,i_k\ge 0}\lambda_{i_1}\cdots\lambda_{i_\ell}\frac{AC_{i_{\ell+1}}^{(\ell)}}{\Pi(i_{\ell+1})}
 =0,
\end{equation}
which is the desired formula.\end{proof}

\begin{proof}[\bf Proof of Theorem \ref{Theorem1}.]
From Lemma \ref{prp0h-r}, we have
$$AC_{m}^{(\ell)}=-\Pi(m)\sum_{i=0}^{m-1}\frac{AC_{i}^{(\ell)}}{\Pi(i)}D_\ell(m-i)$$
with $AC_{0}^{(\ell)}=1,$ where
\begin{equation}
D_\ell(e)=\sum_{i_1+\cdots+i_\ell=e\atop i_1,\dots,i_\ell\ge 0}\lambda_{i_1}\cdots\lambda_{i_\ell},
\end{equation}
which is Theorem \ref{Theorem1}.
\end{proof}

\begin{proof}[\bf Proof of Theorem \ref{Theorem2}.]
Denote by \begin{equation}\label{denoteby}
h(z)=\bigl(f(z)\bigr)^\ell,
\end{equation}
 where
$$f(z)=\sum_{n=0}^\infty\lambda_n z^n.$$
Since by (\ref{HT}), the definition of the  Hasse-Teichm\"uller derivative, we have
\begin{align*}
\left.H^{(i)}(f(z))\right|_{z=0}&=\left.\sum_{n=i}^\infty \lambda_n\binom{n}{i}z^{n-i}\right|_{z=0}\\
&=\lambda_i.
\end{align*}
Then applying  the product rule of the Hasse-Teichm\"uller derivative in Lemma \ref{productrule2}, we get
\begin{equation}\label{h-ell}
\begin{aligned}
\left.H^{(e)}(h(z))\right|_{z=0}&=\sum_{i_1+\cdots+i_\ell=e\atop i_1,\dots,i_\ell\ge0}
\left.H^{(i_1)}(f(z))\right|_{z=0}\cdots\left.H^{(i_\ell)}(f(z))\right|_{z=0}\\
&=\sum_{i_1+\cdots+i_\ell=e\atop i_1,\dots,i_\ell\ge 0}\lambda_{i_{1}}\cdots\lambda_{i_{\ell}} \\
&:=D_\ell(e).
\end{aligned}
\end{equation}
Since by (\ref{denoteby}) and  (\ref{def-hapnf}) \begin{equation*}
\frac{1}{h(z)}=\frac{1}{(f(z))^\ell}=\sum_{n=0}^\infty\frac{AC_n^{(\ell)}}{\Pi(n)}z^n,
\end{equation*}
we have
\begin{equation}\label{above1}
H^{(m)}\left.\left(\frac{1}{h(z)}\right)\right|_{z=0}
=H^{(m)}\left.\left( \sum_{n=0}^\infty\frac{AC_n^{(\ell)}}{\Pi(n)}z^n \right)\right|_{z=0}=\frac{AC_m^{(\ell)}}{\Pi(m)}.\end{equation}
And from  (\ref{quotientrule1}), the quotient rule of the Hasse-Teichm\"uller derivative, and (\ref{h-ell}), we get
\begin{equation}\label{above2}\begin{aligned}
H^{(m)}\left.\left(\frac{1}{h(z)}\right)\right|_{z=0}
&=\sum_{k=1}^m\left.\frac{(-1)^k}{h^{k+1}}\right|_{z=0}
\sum_{e_1+\cdots+e_k=m\atop e_1,\dots,e_k\ge 1}\left.H^{(e_1)}(h(z))\right|_{z=0}\cdots\left.H^{(e_k)}(h(z))\right|_{z=0}\\
&=\sum_{k=1}^m(-1)^{k}\sum_{e_1+\cdots+e_k=m\atop e_1,\dots,e_k\ge1}D_\ell(e_1)\cdots D_\ell(e_k).
\end{aligned}\end{equation}
Comparing (\ref{above1}) and (\ref{above2}) we get the desired formula.
\end{proof}

\begin{proof}[\bf Proof of Theorem \ref{Theorem3}.]
From the  geometric series expansion, we have
\begin{eqnarray}\label{geometric2}
S(z)=\frac{1}{ 1+ (f(z)-1)} =\sum_{j=0}^{\infty}(-1)^j (f(z)-1)^j.
\end{eqnarray}
And by (\ref{def-Sz}) and the definition of the  Hasse-Teichm\"uller derivative (\ref{HT}), we get
$$H^{(m)}\left.(S(z))\right|_{z=0}
=H^{(m)}\left.\left( \sum_{n=0}^\infty\frac{AC_n}{\Pi(n)}z^n \right)\right|_{z=0}=\frac{AC_m}{\Pi(m)}.$$
Then applying  the Hasse-Teichm\"{u}ller derivative $H^{(m)}$ of order $m\geq1 $ to both sides of
(\ref{geometric2}), we get
\begin{equation}\label{proof1}
\frac{AC_m}{\Pi(m)} =\sum_{j=1}^{\infty} (-1)^j H^{(m)}(g^j)|_{z=0},
\end{equation}
where
\begin{equation}\label{g} g:=f(z)-1=\sum_{i=1}^{\infty}\lambda_iz^i.\end{equation}
Lemma \ref{quotientrule3} yields
$$\begin{aligned}
H^{(m)}(g^j)|_{z=0}&=\sum_{k=1}^{j}g^{j-k}
\sum_{\substack{ i_1,\ldots, i_m \geq 0 \\ i_1 +\cdots+ i_m =k \\ i_1 +2i_2 + \cdots+ mi_m =m } }
 \frac{j(j-1) \cdots (j-k+1)}{i_1 ! \cdots i_m !} \\
&\quad\times
  \left.(H^{(1)}(g))^{i_1}  \cdots (H^{(m)}(g))^{i_m}\right|_{z=0}.
\end{aligned}$$
By (\ref{g}), we have $g(0)=0$  and $g^{j-k}|_{z=0}=0$ if $j\neq k$, thus the right hand side of the above equality equals to
$$\sum_{\substack{ i_1, \ldots, i_m \geq 0 \\
i_1 +\cdots+ i_m =j \\
i_1 +2i_2 + \cdots+ mi_m =m } }
 \frac{j!}{i_1 ! \cdots  i_m!}
 \lambda_1^{i_1}  \lambda_2^{i_2}  \cdots \lambda_m^{i_m}.$$
Substituting to (\ref{proof1}) and also noticing  that for $j>m, $ the summation index of  the above sum
 becomes empty thus the sum equals to 0,
we get the desired formula.
\end{proof}

\begin{proof}[\bf Proof of Theorem \ref{Theorem4}.]
At this stage, we show that the proof of \cite[Theorem 2]{HKo}
which based on the inductive method can also be applied to our situation.

Denote by $A_{m}^{(\ell)}=\frac{(-1)^m AC_{m}^{(\ell)}}{\Pi(m)}.$ Then, we shall prove that for any $m\ge 1$
\begin{equation}
A_{m}^{(\ell)}=\left|
\begin{array}{ccccc}
D_\ell(1)&1&&&\\
D_\ell(2)&D_\ell(1)&&&\\
\vdots&\vdots&\ddots&1&\\
D_\ell(n-1)&D_\ell(n-2)&\cdots&D_\ell(1)&1\\
D_\ell(n)&D_\ell(n-1)&\cdots&D_\ell(2)&D_\ell(1)
\end{array}
\right|.
\label{aNnr}
\end{equation}
When $m=1$, (\ref{aNnr}) is valid, because  by Corollary \ref{Corollary2} we have
$$
AC_1=(-1)\Pi(1)D_\ell(1).
$$
Assume that (\ref{aNnr}) is valid up to $m-1$. Notice that
by  Corollary \ref{Corollary1}, we have
$$A_{m}^{(\ell)}=\sum_{i=1}^{m}(-1)^{i-1}A_{m-i}^{(\ell)}D_\ell(i).$$
Thus, by expanding the first row of the right-hand side (\ref{aNnr}), it is equal to
\begin{align*}
&D_\ell(1)A_{m-1}^{(\ell)}-\left|
\begin{array}{ccccc}
D_\ell(2)&1&&&\\
D_\ell(3)&D_\ell(1)&&&\\
\vdots&\vdots&\ddots&1&\\
D_\ell(m-1)&D_\ell(m-3)&\cdots&D_\ell(1)&1\\
D_\ell(m)&D_\ell(m-2)&\cdots&D_\ell(2)&D_\ell(1)
\end{array}
\right|\\
&=D_\ell(1)A_{m-1}^{(\ell)}-D_\ell(2)A_{m-2}^{(\ell)}\\
&\qquad\qquad\;\; +\left|
\begin{array}{ccccc}
D_\ell(3)&1&&&\\
D_\ell(4)&D_\ell(1)&&&\\
\vdots&\vdots&\ddots&1&\\
D_\ell(m-1)&D_\ell(m-4)&\cdots&D_\ell(1)&1\\
D_\ell(m)&D_\ell(m-3)&\cdots&D_\ell(2)&D_\ell(1)
\end{array}
\right|\\
&=D_\ell(1)A_{m-1}^{(\ell)}-D_\ell(2)A_{m-2}^{(\ell)}+\cdots
+(-1)^{m-2}\left|
\begin{array}{cc}
D_\ell(m-1)&1\\
D_\ell(m)&D_\ell(1)
\end{array}
\right|\\
&=\sum_{i=1}^m(-1)^{i-1}D_\ell(i)A_{m-i}^{(\ell)}=A_{m}^{(\ell)}.
\end{align*}
Note that $A_{1}^{(\ell)}=D_\ell(1)$ and $A_{0}^{(\ell)}=1$.
\end{proof}

\section*{Acknowledgement} The authors are enormously grateful to the anonymous referee for his/her very careful
reading of this paper, and for his/her many valuable and detailed suggestions. 
Su Hu is supported by Guangdong Basic and Applied Basic Research Foundation (No. 2020A1515010170).  Min-Soo Kim is supported by the National Research Foundation of Korea(NRF) grant funded by the Korea government(MSIT) (No. 2019R1F1A1062499).

\end{document}